\documentclass[11pt, oneside]{article}   	
\usepackage{geometry}                		
\usepackage{graphicx}				
\usepackage{amssymb,amsthm,xcolor}					
\usepackage{natbib}
\usepackage{amsmath}
\usepackage{epstopdf}
\usepackage{bbm}
\usepackage{mathtools}

\newcommand{\abs}[1]{\left\vert#1\right\vert}

\newcommand{\norm}[1]{\big\Vert#1\big\Vert}

\numberwithin{equation}{section}
\theoremstyle{plain}
\newtheorem{thm}{Theorem}[section]
\theoremstyle{plain}
\newtheorem{lem}{Lemma}[section]
\theoremstyle{remark}
\newtheorem{rem}{Remark}[section]
\theoremstyle{plain}
\newtheorem{ass}{Assumption}[section]

\begin{document}

\title{Adaptive posterior convergence rates in non-linear latent variable models}

\author{
Shuang Zhou\footnote{Corresponding author: shuang.zhou@stat.fsu.edu},  \quad Debdeep Pati\\
{\em Department of Statistics, Florida State University}\\
Anirban Bhattacharya  \\
{\em Department of Statistics, Texas A\&M University}\\
David Dunson \\
{\em Department of Statistical Science,  Duke University}
}

%
%

\maketitle

\begin{center}
\textbf{Abstract}
\end{center}
Non-linear latent variable models have become increasingly popular in a variety of applications. However, there has been little study on theoretical properties of these models. In this article, we study rates of posterior contraction in univariate density estimation for a class of non-linear latent variable models where unobserved $\mbox{U}(0,1)$ latent variables are related to the response variables via a random non-linear regression with an additive error. Our approach relies on characterizing the space of densities induced by the above model as kernel convolutions with a general class of continuous mixing measures. The literature on posterior rates of contraction in density estimation almost entirely focuses on finite or countably infinite mixture models. We develop approximation results for our class of continuous mixing measures. Using an appropriate Gaussian process prior on the unknown regression function, we obtain the optimal frequentist rate up to a logarithmic factor under standard regularity conditions on the true density.
\vspace*{.3in}

\section{Introduction}

Latent variable models are popular in statistics and machine learning for dimension reduction,  parsimonious dependence modeling and data visualization. Linear latent variable models, such as factor models or probabilistic principal components, assume a linear relationship between the observed and latent variables. A number of {\em non-linear} latent variable models have been proposed in the literature for structured dimension reduction and manifold learning; example include the generative tomographic mapping (GTM; \cite{bishop1998gtm,bishop1998developments}) and the Gaussian process latent variable model (GP-LVM; \cite{lawrence2004gaussian,lawrence2005probabilistic,lawrence2007hierarchical,ferris2007wifi}). 
These models flexibly model the relationship between observed and latent variables, notably using Gaussian process (GP) priors. In spite of their empirical success, a general theoretical framework studying the properties of the induced density of the data after marginalizing out the latent variables seems lacking.  

\cite{kundu2011bayesian} proposed an NL-LVM (non linear latent variable model) approach for univariate density estimation in which the response variables are modeled as unknown functions (referred to as the \emph{transfer function}) of uniformly distributed latent variables with an additive Gaussian error. Operationally similar to a univariate GP-LVM model, the latent variable specification allows straightforward posterior computation via conjugate posterior updates. Since inverse c.d.f. transforms of uniform random variables can generate draws from any distribution, a prior with large support on the space of transfer functions can approximate draws from any continuous distribution function arbitrarily closely. One can also conveniently center the non-parametric model on a parametric family by centering the prior on the transfer function on a parametric class of quantile (or inverse c.d.f.) functions $\{F_{\theta}^{-1} ~:~ \theta \in \Theta\}$. While such centering on parametric guesses can be achieved in Dirichlet process (DP; \cite{ferguson1973bayesian,ferguson1974prior}) mixture models by appropriate choice of the base measure $G_0$, posterior computation becomes complicated unless the base measure is conjugate to the kernel $\mathcal{K}$. 

Although there is an increasingly rich literature on asymptotic properties of Bayesian density estimation, this literature mainly focuses on discrete mixture models that have a fundamentally different form from the NL-LVM models.  Hence, it is unclear what types of asymptotic properties NL-LVMs have for density estimation, and technical tools developed in the existing literature cannot be fully utilized to study this problem.  Our focus is on closing this gap focusing in particular on studying how the posterior measure for the unknown density concentrates around the true density $f_0$ as the sample size $n$ increases.  Assuming $f_0$ belongs to a H\"older space of univariate functions with smoothness $\beta$, it is known that the minimax optimal rate of convergence for an estimate of the density is $n^{-\beta/(2\beta +1)}$.  Assuming the prior for the unknown density is induced through a discrete mixture of exponential power distributions, \cite{kruijer2010adaptive} showed that the posterior measure for the density concentrates at the optimal rate up to a logarithmic factor.  Their result shows rate adaptivity to any degree of smoothness of the true density $f_0$, generalizing previous results, such as posterior consistency \cite{ghosal1999posterior} or optimal rates for a particular smoothness level \cite{ghosal2007posterior}. 
We seek to obtain an adaptive rate result for a class of NL-LVM models, and in the process significantly advance technical understanding of this relatively new class of models.

The main contributions of this article are as follows.  We provide an accurate characterization of the prior support in terms of kernel convolution with a class of continuous mixing measures. We provide conditions for the mixing measure to admit a density with respect to Lebesgue measure and show that the prior support of the NL-LVM is at least as large as that of DP mixture models. We then develop approximation results for the above class of continuous mixing measures, and show adaptive convergence rates.  This involves some novel issues and technical details relative to the existing literature.

The rest of the article is organized as follows. We introduce relevant notations and terminologies in Section \ref{sec:notn}. To make the article self-contained, we also provide a brief background on Gaussian process priors. In Section \ref{sec:assf0}, we formulate our assumptions on the true density $f_0$ and in the following section, we describe the NL-LVM model and relate it to convolutions. We state our main theorem on convergence rates in Section \ref{sec:compact}. Section \ref{sec:auxiliary} provides auxiliary results and Section \ref{main} proves the main theorem of posterior concentration rate. We discuss some implications of our results and outlines possible future directions Section \ref{sec:discussion}.

\section{Notations}\label{sec:notn}

Throughout the article, $Y_1,\ldots, Y_n$ are independent and identically
distributed with density $f_0 \in \mathcal{F}$, the set of all densities on $\mathbb{R}$ absolutely continuous with respect to the Lebesgue measure $\lambda$. The supremum and $\mbox{L}_1$-norm are denoted by $\norm{\cdot}_{\infty}$ and $\norm{\cdot}_{1}$, respectively. We let $\norm{\cdot}_{p, \nu}$ denote the norm of
$L_p(\nu)$, the space of measurable functions with $\nu$-integrable $p$th absolute power. For two density functions $f, g \in \mathcal{F}$, let $h$ denote the Hellinger distance defined as $h^2(f, g) = \norm{\sqrt{f}-\sqrt{g}}_{2, \lambda}=\int (f^{1/2} - g^{1/2})^2 d\lambda $, $K(f,g)$ the Kullback-Leibler divergence given by $K(f,g) = \int \log(f/g) f d\lambda$ and $V(f,g)=\int \log(f/g)^2 f d\lambda$. The notation $C[0, 1]$ is used for the space of continuous functions $f : [0, 1] \rightarrow \mathbb{R}$ endowed with the supremum norm. For $\beta >0$, we let $C^{\beta}[0, 1]$ denote the H\"{o}lder space of order $\beta$, consisting of the functions $f \in C[0, 1]$ that have $\lfloor \beta \rfloor$ continuous
derivatives  with the $\lfloor \beta \rfloor$th derivative
$f^{\lfloor \beta \rfloor}$ being Lipschitz continuous of order $\beta -\lfloor \beta \rfloor$. The $\epsilon$-covering number $N(\epsilon,S,d)$ of a semi-metric space $S$ relative to the semi-metric $d$ is the minimal number of balls of radius $\epsilon$ needed to cover $S$. The logarithm of the
covering number is referred to as the entropy.  By near-optimal rate of convergence we mean optimal rate of convergence slowed down by a logarithmic factor.

We write ``$\precsim$'' for inequality up to a constant multiple. Let \linebreak $\phi(x) = (2\pi)^{-1/2}\exp(-x^2/2)$ denote the standard normal density, and let $\phi_{\sigma}(x) = (1/\sigma) \phi(x/\sigma)$. Let an asterisk denote a convolution e.g., $(\phi_{\sigma} * f)(y) = \int \phi_{\sigma}(y - x)f(x)dx$, and let $\phi_{\sigma}^{(i)} * f$ denote the $i$-fold convolution.
The support of a density $f$ is denoted by supp($f$).

We briefly recall the definition of the RKHS of a Gaussian process prior; a detailed review can be found in \cite{van2008reproducing}. A Borel measurable random element $W$ with values in a separable Banach space $(\mathbb{B}, \norm{\cdot})$ (e.g., $C[0,1]$) is called Gaussian if the random variable $b^*W$ is normally distributed for any element $b^* \in \mathbb{B}^*$, the dual space of $\mathbb{B}$. The reproducing kernel
Hilbert space (RKHS) $\mathbb{H}$ attached to a zero-mean Gaussian process $W$ is
defined as the completion of the linear space of functions $t \mapsto EW(t)H$
relative to the inner product
\begin{eqnarray*}
\langle \mbox{E} W(\cdot)H_1; \mbox{E} W(\cdot )H_2\rangle_{\mathbb{H}} = \mbox{E} H_1H_2,
\end{eqnarray*}
where $H, H_1$ and $H_2$ are finite linear combinations of the form $\sum_{i}a_{i}W(s_i)$
with $a_i \in \mathbb{R}$ and $s_i$ in the index set of $W$.

Let $W = (W_t: t \in \mathbb{R})$ be a Gaussian process with squared exponential covariance kernel. The spectral measure $m_{w}$ of $W$ is absolutely continuous with respect to the Lebesgue measure $\lambda$ on $\mathbb{R}$ with the Radon-Nikodym derivative given by
\begin{eqnarray*}
\frac{dm_{w}}{d\lambda}(x) = \frac{1}{2\pi^{1/2}}e^{-x^2/4}.
\end{eqnarray*}

Define a scaled Gaussian process $W^a=(W_{at}: t \in [0,1])$, viewed as a map in $C[0,1]$. Let $\mathbb{H}^a$ denote the RKHS of $W^a$, with the corresponding norm $\norm{\cdot}_{\mathbb{H}^a}$. The unit ball in the RKHS is denoted $\mathbb{H}^a_1$. 

Throughout the paper, $C, C_1, C_2, \dots$ denote global constants whose value may change one line to another.

\section{Assumptions on the true density}\label{sec:assf0}

It has been widely recognized that one needs certain smoothness assumptions and tail conditions on the true density $f_0$ to derive posterior convergence rates. We make the following assumptions in our case:
\begin{ass}\label{ass:1}
Assume $\log{f_0}\in C^{\beta}[0, 1]$. Let $l_j(x)=\frac{d^j}{dx^j}\log f_0(x)$ be the derivatives for $j=1, \dots, r$ with $r= \lfloor \beta \rfloor$. For any $\beta >0$, assume that there exists a constant $L > 0$ such that
\begin{eqnarray}\label{eq:tails}
|l_r(x) - l_r(y)| \le L|x-y|^{\beta - r},
\end{eqnarray}
for all $x \neq y$.
\end{ass}

\begin{ass}\label{ass:2}
Assume $f_0$ is compactly supported on $[a_0,b_0]$, for $-\infty < a_0 <b_0 < \infty$, and that there exists some interval $[a,b] \subset [a_0, b_0]$ such that $f_0$ is nondecreasing on $[a_0, a]$, bounded away from $0$ on $[a,b]$ and non-increasing on $[b, b_0]$.
\end{ass}
Assumption \ref{ass:1} is useful in simplifying expressions for $f_0$ as convolutions with a given density, providing a key piece in our theoretical developments. Similar assumption on the local smoothness appeared in \cite{kruijer2010adaptive}, while in our case the global smoothness assumption is sufficient since $f_0$ is assumed to be compactly supported. 
 Assumption \ref{ass:2} guarantees that for every $\delta > 0$, there exists a constant $C > 0$ such that $f_0 * \phi_{\sigma} \geq Cf_0$ for every $\sigma < \delta$.  

\section{The NL-LVM model}\label{sec:model}

Consider the nonlinear latent variable model,
\begin{align}\label{eq:nl-lvm}
y_i &= \mu(\eta_i) + \epsilon _i, \quad \epsilon_i \sim \mbox{N}(0, \sigma^2), \, (i=1, \ldots, n)\\
\mu &\sim \Pi_{\mu}, \quad \sigma \sim \Pi_{\sigma},\quad \eta_i \sim \mbox{U}(0,1),
\end{align}
where $\eta_i$'s are latent variables, $\mu \in C[0, 1]$ is a \emph{transfer function} relating the latent variables to the observed variables and $\epsilon_i$ is an idiosyncratic error. The density of $y$ conditional on the transfer function $\mu$ and scale $\sigma$ is obtained on marginalizing out the latent variable as
\begin{eqnarray}\label{eq:gpt_def}
f(y; \mu, \sigma) \stackrel{\text{def}}{=} f_{\mu, \sigma}(y)= \int_{0}^{1}\phi_{\sigma}(y-\mu(x))dx.
\end{eqnarray}
Define a map $g: C[0,1] \times [0,\infty) \to \mathcal{F}$ with $g(\mu, \sigma) = f_{\mu, \sigma}$.
One can induce a prior $\Pi$ on $\mathcal{F}$ via the mapping $g$ by placing independent priors
$\Pi_{\mu}$ and $\Pi_{\sigma}$ on $C[0,1]$ and $[0, \infty)$ respectively, with $\Pi = (\Pi_{\mu} \otimes \Pi_{\sigma}) \circ g^{-1}$. \cite{kundu2011bayesian} assumed a Gaussian process prior with squared exponential covariance kernel on $\mu$ and an inverse-gamma prior on $\sigma^2$.

It is not immediately clear whether the class of densities $f_{\mu, \sigma}$ in the range of $g$ encompass a large subset of the density space. The following intuition relates the above class with continuous convolutions which plays a key role in our proofs. Let $f_0$ be a continuous density with cumulative distribution function $F_0(t) = \int_{-\infty}^{t} f_0(x) dx$. Assume $f_0$ to be non-zero almost everywhere within its support, so that $F_0 : \mbox{supp}(f_0) \to [0,1]$ is strictly monotone and hence has an inverse $F_0^{-1} : [0,1] \to \mbox{supp}(f_0)$ satisfying $F_0\{F_0^{-1}(t)\} = t$ for all $t \in \mbox{supp}(f_0)$.
Letting $\mu_0(x) = F_0^{-1}(x)$, one obtains
\begin{eqnarray}\label{eq:conv}
f_{\mu_0, \sigma}(y) = \int_{0}^{1} \phi_{\sigma}(y-F_0^{-1}(x))dx = \int_{-\infty}^{\infty} \phi_{\sigma}(y-t) f_0(t) dt,
\end{eqnarray}
where the second equality follows from the change of variable theorem. Thus, $f_{\mu_0,\sigma}(y) = \phi_{\sigma}*f_0$, i.e., $f_{\mu_0, \sigma}$ is the convolution of $f_0$ with a normal density having mean $0$ and standard deviation $\sigma$. It is well known that the convolution $\phi_{\sigma}*f_0$ can approximate $f_0$ arbitrary closely as the bandwidth $\sigma \to 0$ in the sense that for $f_0 \in L_p(\lambda)$ with $p \geq1$, $\norm{\phi_{\sigma}*f_0 - f_0}_{p, \lambda} \to 0$ as $\sigma \to 0$. For Holder-smooth functions, the order of approximation can be characterized in terms of the smoothness. If $f_0 \in C^{\beta}[0, 1]$ with $\beta \le 2$, it follows from standard Taylor series expansion that $|| \phi_{\sigma}*f_0 - f_0||_{\infty} = O(\sigma^{\beta})$. For $\beta > 2$, a similar Taylor series expansion yields a sub-optimal error $|| \phi_{\sigma}*f_0 - f_0||_{\infty} = O(\sigma^{2})$. In this case, we can refine the approximation by convoluting with a sequence of functions $f_j$ constructed by the procedure,
\begin{eqnarray}\label{eq:constrction}
f_{j+1} = f_0 - \triangle_{\sigma} f_j, \quad \triangle_{\sigma}f_j = \phi_{\sigma}*f_j - f_j, \quad j \ge 1.
\end{eqnarray} 
For $f_0 \in C^{\beta}[0,1]$ with $\beta \in (2j, 2j+2]$ we have $\norm{\phi_{\sigma}*f_j - f_0}_{\infty} = O(\sigma^{\beta})$ \citep{kruijer2010adaptive}. Although the $f_j$s need not be non-negative in general, we show that they are non-negative on $\mbox{supp}(f_0)$ when $f_0$ is compactly supported. It can be additionally shown that the normalizing constant is $1 + O(\sigma^{\beta})$; let $h_j$ denote the density obtained by normalizing $f_j$. We then approximate $f_0$ by $\phi_{\sigma} * h_{\beta}$, where $h_{\beta} = h_j$ for $\beta \in (2j, 2j+2]$.


Let $\tilde{\lambda}$ denote the Lebesgue measure on $[0,1]$, or equivalently, the $\mbox{U}[0,1]$ distribution. For any measurable function $\mu : [0,1] \to \mathbb{R}$, let $\nu_{\mu}$ denote the induced measure on $(\mathbb{R}, \mathcal{B})$, with $\mathcal{B}$ denoting the Borel sigma-field on $\mathbb{R}$. Then, for any Borel measurable set $B$, $\nu_{\mu}(B) = \tilde{\lambda}(\mu^{-1}(B))$, where $\mu^{-1}(B) = \{x \in [0,1] ~ : ~ \mu(x) \in B\}$. By the change of variable theorem for induced measures,
\begin{eqnarray}\label{eq:dpgp}
\int_{0}^{1}\phi_{\sigma}(y-\mu(x))dx = \int \phi_{\sigma}(y-t) d\nu_{\mu}(t),
\end{eqnarray}
so that $f_{\mu, \sigma}$ can be expressed as a kernel mixture form 
with mixing distribution $\nu_{\mu}$. It turns out that this mechanism of creating random distributions is very general. Depending on the choice of $\mu$, one can create a large variety of mixing distributions based on this specification. For example, if $\mu$ is a strictly monotone function, then $\nu_{\mu}$ is absolutely continuous with respect to the Lebesgue measure, while choosing $\mu$ to be a step function, one obtains a discrete mixing distribution. However, it is easier to place a prior on $\mu$ supported on the space of continuous functions $C[0, 1]$ without further shape restrictions and Theorem \ref{thm:support} assures us that this specification leads to large $L_1$ support on the space of densities.

Suppose the prior $\Pi_{\mu}$ on $\mu$ has full sup-norm support on $C[0,1]$ so that $\mbox{Pr}(\norm{\mu - \mu^*}_{\infty} < \epsilon) > 0$ for any $\epsilon > 0$ and $\mu^* \in C[0,1]$, and the prior $\Pi_{\sigma}$ on $\sigma$ has full support on $[0, \infty)$. If $f_0$ is compactly supported so that the quantile function $\mu_0 \in C[0,1]$, then it can be shown that under mild conditions, the induced prior $\Pi$ assigns positive mass to arbitrarily small $L_1$ neighborhoods of any density $f_0$. We summarize the above discussion in the following theorem, with a proof provided in the appendix.

\begin{thm}\label{thm:support}
If $\Pi_{\mu}$ has full sup-norm support on $C[0,1]$ and $\Pi_{\sigma}$ has full support on $[0, \infty)$, then the $L_1$ support of the induced prior $\Pi$ on $\mathcal{F}$ contains all densities $f_0$ which have a finite first moment and are non-zero almost everywhere on their support.
\end{thm}
\begin{rem}
The conditions of Theorem \ref{thm:support} are satisfied for a wide range of Gaussian process priors on $\mu$ (for example, a GP with a squared exponential or Mat\'{e}rn covariance kernel).
\end{rem}
\begin{rem}
When $f_0$ has full support on $\mathbb{R}$, the quantile function $\mu_0$ is unbounded near $0$ and $1$, so that $\norm{\mu_0}_{\infty} = \infty$. However, $\int_{0}^{1} \abs{\mu_0(t)} dt = \int_{\mathbb{R}} \abs{x} f_0(x) dx$, which implies that $\mu_0$ can be identified as an element of $L_1[0,1]$ if $f_0$ has finite first moment. Since $C[0,1]$ is dense in $L_1[0,1]$, the previous conclusion regarding $L_1$ support can be shown to hold in the non-compact case too.
\end{rem}

\section{The main theorem}\label{sec:compact}

We consider the case where $f_0$ satisfies Assumption \ref{ass:1} and Assumption \ref{ass:2}. For $\beta$-H\"{o}lder density $f_0$, we consider density $h_{\beta}$ as defined after expression (\ref{eq:constrction}).  Denote $\mu_0$ the quantile function $F_{h_{\beta}}^{-1}: [0,1] \to [a_0, b_0]$, a continuous monotone function inheriting the smoothness of $h_{\beta}$. Note that $h_{\beta}$ has the same smoothness of $f_0$ based on the construction of $f_j$, therefore with the fundamental theorem of calculus it is easy to see that $\mu_0 \in C^{\beta+1}[0,1]$.

We now mention our choices for the prior distributions $\Pi_{\mu}$ and $\Pi_{\sigma}$.
\begin{ass}\label{ass:1p}
We assume $\mu$ follows a centered and rescaled Gaussian process denoted by $\mbox{GP}(0, c^A)$, where $A$ denotes the rescaled parameter, and assume $A$ is a density satisfying for $a>0$,
\begin{eqnarray}\label{tails}
C_1a^p\exp{(-D_1a \log^q a)} \le g(a) \le C_2a^p\exp{(-D_2a \log^q a)}
\end{eqnarray}
for positive constants $C_1$, $C_2$, $D_1$, $D_2$, nonnegative constant $p$ and $q$, and every sufficiently large $a>0$.
\end{ass}
\begin{ass}\label{ass:2p}
We assume $\sigma \sim \mbox{IG}(a_{\sigma}, b_{\sigma})$.
\end{ass}
Note that contrary to the usual conjugate choice of an inverse-Gamma prior for $\sigma^2$, we have assumed an inverse-Gamma prior for $\sigma$.  This enables one to have slightly more prior mass near zero compared to an inverse-Gamma prior for $\sigma^2$, leading to the optimal rate of posterior convergence. Refer also to \cite{kruijer2010adaptive} for a similar prior choice for the bandwidth of the kernel in discrete location-scale mixture priors for densities.

We state below the main theorem of posterior convergence rates.
 \begin{thm}\label{thm:compact}
If $f_0$ satisfies Assumption \ref{ass:1} and the priors $\Pi_\mu$ and $\Pi_{\sigma}$ are as in
Assumptions \ref{ass:1p} and \ref{ass:2p} respectively, the best obtainable rate of posterior convergence relative to Hellinger metric $h$ is
\begin{eqnarray}\label{eq:optrate}
\epsilon_{n} = \max(\tilde{\epsilon}_n, \bar{\epsilon}_n),
\end{eqnarray}
where $\tilde{\epsilon}_n=n^{-\frac{\beta}{2\beta+1}}(\log n)^{t_1}$, $\bar{\epsilon}_n= n^{-\frac{\beta}{2\beta+1}}(\log n)^{t_2}$, with nonnegative constants $t_1=\beta(2\vee q)/(2\beta +1)$, $t_2= t_1+1$.
\end{thm}
Unlike the treatment in discrete mixture models \citep{ghosal2007posterior} where a compactly supported density is approximated with a discrete mixture of normals, the main trick here is to approximate the true density $f_0$ by the convolution $\phi_{\sigma}*f_0$ and allow the prior on the transfer function to appropriately concentrate around the true quantile function $\mu_0 \in C[0,1]$.

\section{Auxiliary results}\label{sec:auxiliary}
To guarantee that the above scheme leads to the optimal rate of convergence, we first derive sharp bounds for the Hellinger distance between $f_{\mu_1, \sigma_1}$ and $f_{\mu_2, \sigma_2}$ for $\mu_1, \mu_2 \in C[0, 1]$ and $\sigma_1, \sigma_2 > 0$. We summarize the result in the following Lemma \ref{lem:hellinger}.
\begin{lem}\label{lem:hellinger}
For $\mu_1, \mu_2 \in C[0, 1]$ and $\sigma_1, \sigma_2 > 0$,
\begin{eqnarray}
h^2(f_{\mu_1,\sigma_1}, f_{\mu_2, \sigma_2}) \leq 1- \sqrt{\frac{2\sigma_1\sigma_2}{\sigma_1^2 + \sigma_2^2}}\exp\bigg\{-\frac{\norm{\mu_1 - \mu_2}_{\infty}^2}{4(\sigma_1^2 + \sigma_2^2)}\bigg\}.
\end{eqnarray}
\end{lem}
\begin{proof}
Note that by H\"{o}lder's inequality,
\begin{eqnarray*}
f_{\mu_1, \sigma_1}(y)f_{\mu_2, \sigma_2}(y) \geq \bigg\{\int_{0}^{1} \sqrt{\phi_{\sigma_1}(y - \mu_1(x))} \sqrt{\phi_{\sigma_2}(y - \mu_2(x))}dx\bigg\}^2.
\end{eqnarray*}
Hence,
\begin{eqnarray*}
h^2(f_{\mu_1,\sigma_1}, f_{\mu_2, \sigma_2}) &\leq& \int\bigg[\int_{0}^{1}\phi_{\sigma_1}(y-\mu_1(x))dx  + \int_{0}^{1}\phi_{\sigma_2}(y-\mu_2(x))dx \\
&-&2\int_{0}^{1}\sqrt{\phi_{\sigma_1}(y - \mu_1(x))} \sqrt{\phi_{\sigma_2}(y - \mu_2(x))}dx\bigg]dy.
\end{eqnarray*}
By changing the order of integration (applying Fubini's theorem since the function within the integral is jointly integrable) we get,
\begin{eqnarray*}
h^2(f_{\mu_1,\sigma_1}, f_{\mu_2, \sigma_2}) &\leq& \int_{0}^{1}h^2(f_{\mu_1(x),\sigma_1}, f_{\mu_2(x), \sigma_2})dx \\
&=& \int_{0}^{1}\bigg[1- \sqrt{\frac{2\sigma_1\sigma_2}{\sigma_1^2 + \sigma_2^2}}\exp\bigg\{-\frac{(\mu_1(x) - \mu_2(x))^2}{4(\sigma_1^2 + \sigma_2^2)}\bigg\}\bigg]dx\\
&\leq& 1- \sqrt{\frac{2\sigma_1\sigma_2}{\sigma_1^2 + \sigma_2^2}}\exp\bigg\{-\frac{\norm{\mu_1 - \mu_2}_{\infty}^2}{4(\sigma_1^2 + \sigma_2^2)}\bigg\}.
\end{eqnarray*}

\end{proof}

\begin{rem}
When $\sigma_1 = \sigma_2  = \sigma$, $h^2(f_{\mu_1,\sigma}, f_{\mu_2, \sigma}) \leq 1 - \exp\big\{\norm{\mu_1 -\mu_2}_{\infty}^2/ 8 \sigma^2\big\}$, which implies that $h^2(f_{\mu_1,\sigma}, f_{\mu_2, \sigma}) \precsim \norm{\mu_1 -\mu_2}_{\infty}^2/\sigma^2$.
\end{rem}

\begin{rem}
The standard inequality $h^2(f_{\mu_1,\sigma_1}, f_{\mu_2, \sigma_2}) \leq \norm{f_{\mu_1,\sigma_1}- f_{\mu_2, \sigma_2}}_1$ relating the Hellinger distance to the total variation distance leads to the cruder bound
\begin{eqnarray*}
h^2(f_{\mu_1,\sigma_1}, f_{\mu_2, \sigma_2}) \leq C_1 \frac{\norm{\mu_1 -\mu_2}_{\infty}}{(\sigma_1 \wedge \sigma_2)} + C_2\frac{|\sigma_2 - \sigma_1|}{(\sigma_1 \wedge \sigma_2)},
 \end{eqnarray*}
which is linear in $\norm{\mu_1 -\mu_2}_{\infty}$. This bound is less sharp than what is obtained in Lemma \ref{lem:hellinger} and does not suffice for obtaining the optimal rate of convergence.
\end{rem}

To control the Kullback-Leibler divergence between the true density $f_0$ and the model $f_{\mu, \sigma}$, we derive an upper bound for $\log \norm{\frac{f_0}{f_{\mu, \sigma}}}_{\infty}$ in Lemma \ref{lem:logsup}.
\begin{lem}\label{lem:logsup}
If $f_0$ satisfies Assumption \ref{ass:2},
\begin{eqnarray}
\log \norm{\frac{f_0}{f_{\mu, \sigma}}}_{\infty} \leq C + \frac{\norm{\mu - \mu_0}_{\infty}^2}{\sigma^2}
\end{eqnarray}
for some constant $C > 0$.
\end{lem}
\begin{proof}
Note that
\begin{eqnarray*}
f_{\mu, \sigma}(y) &=& \frac{1}{\sqrt{2\pi}\sigma}\int_{0}^{1}\exp\bigg\{-\frac{(y-\mu(x))^2}{2\sigma^2}\bigg\}dx\\
&\geq&  \frac{1}{\sqrt{2\pi}\sigma}\int_{0}^{1}\exp\bigg\{-\frac{(y-\mu(x))^2}{\sigma^2}\bigg\}dx \exp\bigg\{-\frac{\norm{\mu-\mu_0}_{\infty}^2}{\sigma^2}\bigg\}\\
&\geq& C \phi_{\sigma/\sqrt{2}} * f_0 (y) \exp\bigg\{-\frac{\norm{\mu-\mu_0}_{\infty}^2}{\sigma^2}\bigg\}\\
&\geq& C f_0(y) \exp\bigg\{-\frac{\norm{\mu-\mu_0}_{\infty}^2}{\sigma^2}\bigg\},
\end{eqnarray*}
where the last inequality follows from Lemma 6 of \cite{ghosal2007posterior} since $f_0$ is compactly supported by Assumption \ref{ass:2}. This provides the desired inequality. 
\end{proof}


\begin{lem}\label{lem:approx}
For $\beta \in (2j, 2j+2], j \ge 0$ and $f_j$ constructed by \ref{eq:constrction}, we have the expression $f_{j} = \sum_{i=0}^{j}(-1)^i {j+1 \choose i+1} \phi_{\sigma}^{(i)} f_0$.
\end{lem}

The proof can be found in Appendix \ref{app:combination}. The expression of $f_j$ as a linear combination of true density and the folded convolutions indicates that $f_j$ is as smooth as $f_0$. One can get infinitely differentiable function by convoluting with the kernel, so for the true density with higher regularity degree, we need add the "smoother" function into the approximation $f_j$ to ensure the approximation error remains of order $O(\sigma^{\beta})$.

\begin{lem}\label{lem:convolution}
For any $\beta > 0$, let $f_0$ satisfy Assumption \ref{ass:1} and \ref{ass:2}, integer $j$ be such that for $\beta \in (2j, 2j+2] $, $f_j$ constructed by (\ref{eq:constrction}). For any constant $L$ and all $x \in [a_0, b_0]$, we have 
\begin{eqnarray}\label{eq:convol}
\phi_{\sigma} * f_{\beta} (x) = f_0(x) (1 + O(D(x)\sigma^{\beta})),
\end{eqnarray}
where
\begin{eqnarray*}
D(x) = \sum_{i=1}^{r} c_i {|l_j(x)|}^{\frac{\beta}{i}} + c_{r+1},
\end{eqnarray*}
for nonnegative constants $c_i, i = 1, \dots, r$, and $c_{r+1}$ a multiple of $L$.
\end{lem}
\begin{proof}
Following the proof of Lemma 1 in \cite{kruijer2010adaptive}, for any $x, y \in [a_0, b_0]$,
\begin{eqnarray*}
\log{f_0(y)} \le \log{f_0(x)} + \sum_{i=1}^{r} \frac{l_j(x)}{j!}(y-x)^j + L|y-x|^{\beta},
\end{eqnarray*}
\begin{eqnarray*}
\log{f_0(y)} \ge \log{f_0(x)} + \sum_{i=1}^{r} \frac{l_j(x)}{j!}(y-x)^j - L|y-x|^{\beta}.
\end{eqnarray*}
Define
\begin{eqnarray*}
B^u_{f_0,r}(x,y) = \sum_{i=1}^{r} \frac{l_j(x)}{j!}(y-x)^j + L|y-x|^{\beta},
\end{eqnarray*}
\begin{eqnarray*}
B^l_{f_0,r}(x,y) = \sum_{i=1}^{r} \frac{l_j(x)}{j!}(y-x)^j - L|y-x|^{\beta}.
\end{eqnarray*}
Then we have
\begin{eqnarray*}
e^{B^u_{f_0,r}} \le 1 + B^u_{f_0,r} + \frac{1}{2!}(B^u_{f_0,r})^2 + \dots + M |B^u_{f_0,r}|^{r+1},
\end{eqnarray*}
\begin{eqnarray*}
e^{B^l_{f_0,r}} \ge 1 + B^l_{f_0,r} + \frac{1}{2!}(B^l_{f_0,r})^2 + \dots - M |B^l_{f_0,r}|^{r+1}.
\end{eqnarray*}
where 
\begin{eqnarray*}
M = \frac{1}{(r+1)!} \exp \bigg \{\sup_{x, y  \in [a_0, b_0], x \neq y} (|\sum_{j=1}^{r} \frac{l_j(x)}{j!} (y-x)^j| + L|y-x|^{\beta})\bigg \}.
\end{eqnarray*}
Note that $f_0$ is bounded on $[a_0, b_0]$, we consider the convolution on the whole real line by extending $f_0$ analytically outside $[a_0, b_0]$. For $\beta \in (1, 2], r = 1$ and $x \in (a_0, b_0)$, 
\begin{eqnarray*}
\phi_{\sigma}* f_0(x) &\le& f_0(x) \int e^{B^u_{f_0,r}(x, y)} \phi_{\sigma}(y-x) dy\\
&\le& f_0(x) \int_{\mathbb{R}} \phi_{\sigma}(y-x) [ 1 + L|y-x|^{\beta} + M ( l^2_1(x)(y-x)^2 + L^2|y-x|^{2\beta}) ] dy.
\end{eqnarray*}
Since $l_j(x)$'s are all continuous on $[a_0, b_0]$, there exist finite constants $M_j$ such that $|l_j| \le M_j$ and $|y-x| \le |b_0 - a_0|$. The integral in the last inequality can be bounded by 
\begin{eqnarray*}
\int_{\mathbb{R}} \phi_{\sigma}(y-x) [ 1 + L|y-x|^{\beta} + M ((M_1|b_0 - a_0|)^{2-\beta} |l_1(x)(y-x)|^{\beta} + (L^2|b_0 - a_0|^{\beta} )|y-x|^{\beta}) ] dy
\end{eqnarray*}
Therefore,
\begin{eqnarray*}
\phi_{\sigma}* f_0(x) \le f_0(x) ( 1 + (r_1|l_1(x)|^{\beta} +r_2) \sigma^{\beta}),
\end{eqnarray*}
where $ r_1 = M (M_1|b_0 - a_0|)^{2-\beta}\mu_{\beta}, \  \ r_2 =  (L + ML^2)\mu_{\beta}$. \\
In the other direction, 
\begin{eqnarray*}
\phi_{\sigma}* f_0(x) \ge  f_0(x) \int \phi_{\sigma}(y-x) [ 1 - L|y-x|^{\beta} - M ( l^2_1(x)(y-x)^2 + L^2|y-x|^{2\beta}) ] dy.
\end{eqnarray*}
Thus we achieve (\ref{eq:convol}).

For any $\beta > 2$, suppose $\beta \in (2j, 2j +2], j > 1$. First we calculate $\phi_{\sigma}* f_0$, $\phi^{(2)}_{\sigma}* f_0$, $\dots$, $\phi^{(j)}_{\sigma}* f_0(x)$, by Lemma \ref{lem:approx} to get $\phi_{\sigma}* f_{\beta}(x)$. The calculation of $\phi^{(i)}_{\sigma}* f_0(x)$ is the same as $\phi_{\sigma}* f_0(x)$ except taking the convolution with $\phi_{\sqrt{i}\sigma}$. The terms $\sigma^2$, $\sigma^4$, $\dots$, $\sigma^{2j}$ caused by the factors containing $|y-x|^k, k < \beta$ in $\phi^{(i)}_{\sigma} f_0$ can be canceled out by Lemma \ref{lem:approx}.
For terms containing $|y-x|^k, k \ge \beta$, we take out $|y-x|^{\beta}$ and bound the rest by a certain power of $|l_j(x)|$ or some constant. 
\end{proof}

\begin{lem}\label{lem:tail}
Let $f_0$ satisfy Assumption \ref{ass:1} and \ref{ass:2}. With $A_\sigma = \{x: f_0(x) \ge \sigma^{H}\}$ , we have
\begin{eqnarray}\label{eq:tail} 
\int_{A^c_\sigma}f_0(x)dx = O(\sigma^{2\beta}), \  \   \int_{A^c_\sigma} \phi_{\sigma}*f_{j}(x) dx = O(\sigma^{2\beta}),
\end{eqnarray}
for all non-negative integer $j$, sufficiently small $\sigma$ and sufficiently Large $H$.
\end{lem}
\begin{proof}
Under Assumption \ref{ass:2} there exists $(a, b) \subset [a_0, b_0]$ such that ${A^c_\sigma} \subset [a_0, a) \cup (b, b_0]$ if we choose $\sigma$ sufficiently small, so that $f_0 (x) \le \sigma^{H}$ for $x \in {A^c_\sigma}$. Therefore, $\int_{A^c_\sigma}f_0(x) \le \sigma^{H} |b_0 - a_0| \le O(\sigma^{2\beta})$ if we choose $H \ge 2\beta$. Using Lemma \ref{lem:convolution}, 
\begin{eqnarray*}
\int_{A^c_\sigma} \phi_{\sigma}* f_j(x) dx =  \int_{A^c_\sigma} f_0(x) (1 + O(D(x)\sigma^{\beta})) \le O(\sigma^{H}),
\end{eqnarray*}
with bounded $D(x)$ and $H \ge 2\beta$ it is easy to bound the second integral by $O(\sigma^{2\beta})$.
\end{proof}

\begin{lem}\label{lem:density}
Suppose $f_0(x)$ satisfies Assumption \ref{ass:1} and \ref{ass:2}. For $\beta >2$ and $j$ such that $\beta \in (2j,2j+2]$, we can construct the density $h_{\beta}$ from (\ref{eq:constrction}) and show that $h_{\beta}$ satisfies Lemma \ref{lem:convolution} and Lemma \ref{lem:tail}.
\end{lem}
\begin{proof}
To get the density function we first show that $f_j$ is nonnegative and compute the normalizing constant $\int f_j(x) = 1 + O(\sigma^{\beta})$. Following the proof of Lemma 2 in \cite{kruijer2010adaptive}, we treat $\log f_0$ as a function in $C^2 [0,1]$ and obtain the same form of $\phi_{\sigma}*f_0$ as (\ref{eq:convol}). For small enough $\sigma$ we can find  $\rho_1 \in (0,1)$ very close to $0$ such that
\begin{eqnarray*}
 \phi_{\sigma}* f_0(x) = f_0(x) (1 + O(D^{(2)}(x)\sigma^2)) < f_0(x)(1 + \rho_1),
\end{eqnarray*}
 where $D^{(2)}$ contains $|l_1(x)|$ and $|l_2(x)|$ with certain power, so $D^{(2)}$ is bounded. Then we have
 \begin{eqnarray*}
 f_1(x) = 2f_0(x) - K_{\sigma} f_0(x) > 2f_0(x) - f_0(x) (1 + \rho_1) = f_0(x) (1 - \rho_1).
 \end{eqnarray*}
 Then we treat $\log{f_0}$ as a function with $\beta = 4$, $j = 1$. Similarly, we can get
 \begin{eqnarray*}
 \phi_{\sigma}* f_1(x) = f_0(x) (1 + O(D^{(4)}(x)\sigma^4)),
 \end{eqnarray*}
 where $D^{(4)}$ contains $|l_1(x)|, \dots. |l_4(x)|$. We can find $0 < \rho_2 < \rho_1$ such that $\phi_{\sigma}* f_1(x) < f_0(x) (1 + \rho_2)$, then can get
 \begin{eqnarray*}
 f_2 (x) = f_0(x) - (\phi_{\sigma}* f_1(x) - f_1(x)) > f_0(x)(1 - \rho_1 - \rho_2) > f_0(x)(1 - 2\rho_1).
 \end{eqnarray*}
 Continuing this procedure, we can get $f_j(x) > f_0(x) (1 - j\rho_1)$, with sufficiently small $\sigma$, and $1 - j\rho_1 \in (0,1)$ but very close to 1. Obviously $f_j$ is nonnegative. \\
Now we calculate the normalizing constant for $f_j$. When $\beta < 2$, with Lemma \ref{lem:approx},
\begin{eqnarray*}
\int f_1(x) = \int f_0 - (\phi_{\sigma}* f_0 - f_0) \le \int f_0 + |\int (\phi_{\sigma}*f_0 - f_0)| \le 1 + O(\sigma^{\beta}). 
\end{eqnarray*}
For $\beta \in (2,4]$, 
\begin{eqnarray*}
\int f_2(x) = \int f_0 -  (\phi_{\sigma}* f_1 - f_1) \le \int f_1 + |\int (\phi_{\sigma}*f_1 - f_0)| \le 1 + O(\sigma^{\beta}).
 \end{eqnarray*}
 Then by induction, we have $\int f_j = 1 + O(\sigma^{\beta})$, so that we have the density
\begin{eqnarray}\label{eq:density}
h_{\beta}= \frac{f_j}{1 + O(\sigma^{\beta})}, \quad \beta \in (2j, 2j+2]. 
\end{eqnarray}
 Now to show $h_{\beta}$ satisfying (\ref{eq:convol}), note that 
 \begin{eqnarray*}
 \phi_{\sigma}* h_{\beta} = \frac{\phi_{\sigma}* f_j}{\int f_j} = \frac{f_0(x)(1+ O(D(x)\sigma^{\beta}))}{1+ O(\sigma^{\beta})}.
 \end{eqnarray*}
 Since $D(x)$ is bounded and for sufficiently small $\sigma$ we consider term 
 \begin{eqnarray*}
 \frac{(1+ O(D(x)\sigma^{\beta}))}{1+ O(\sigma^{\beta})} \precsim \frac{(1+ O((D(x)+1)\sigma^{\beta})+ O(D(x)\sigma^{2\beta}))}{1+ O(\sigma^{\beta})} = 1+ O(D(x)\sigma^{\beta}),
 \end{eqnarray*}
which directly leads to the same form as (\ref{eq:convol}), and obviously Lemma \ref{lem:tail} is satisfied.
\end{proof}

\begin{lem}\label{lem:KL}
 Let $f_0$ satisfy Assumption \ref{ass:1} and \ref{ass:2}, and integer $j$ be such that $\beta \in (2j, 2j+2]$. Then we can show that the density $h_{\beta}$ defined by (\ref{eq:density}) satisfies, \begin{eqnarray}\label{eq:KL}
\int f_0(x) \log{\frac{f_0(x)}{\phi_{\sigma}* h_{\beta}(x)}} = O(\sigma^{2\beta}),
\end{eqnarray}
for sufficiently small $\sigma$ and all $x \in [a_0, b_0]$.
\end{lem}
\begin{proof}
Again consider the set $A_\sigma = \{x: f_0(x) \ge \sigma^{H}\}$ with arbitrarily large $H$. We separate the Kullback-Leibler divergence into
\begin{eqnarray*}
\int_{[a_0, b_0]} f_0 \log{\frac{f_0}{\phi_{\sigma}* h_{\beta}}} \le \int_{A_\sigma} \frac{(f_0 - \phi_{\sigma}*h_{\beta})^2}{\phi_{\sigma}* h_{\beta}} +  \int_{A^c_\sigma} f_0 \log{\frac{f_0}{\phi_{\sigma}* h_{\beta}}} +  \int_{A^c_\sigma}(\phi_{\sigma}*h_{\beta} - f_0). 
\end{eqnarray*}
Under Assumption \ref{ass:2} and by Remark 3 in \cite{ghosal1999posterior}, for small enough $\sigma$ there exists constant $C$ such that for all $x \in [a_0, b_0]$, $\phi_{\sigma}*f_0 \ge Cf_0$, especially on set $A^c_{\sigma}$ $f_0$ satisfies $\phi_{\sigma}*f_0 \ge f_0/3$. Also in the proof of Lemma \ref{lem:density} we can find $\rho \in (0,1)$ such that $f_{\beta} > \rho f_0$. Then we have on set $A_\sigma$ with sufficiently small $\sigma$
\begin{eqnarray*}
\phi_{\sigma}*h_{\beta} = \frac{\phi_{\sigma}*f_{\beta}}{1+O(\sigma^{\beta})} \ge \frac{\rho \phi_{\sigma}*f_0}{1 + O(\sigma^{\beta})} \ge K f_0,
\end{eqnarray*}
for some positive and finite constant $K$. Applying Lemma \ref{lem:convolution}, the first integral on the right side of \ref{eq:KL} can be bounded by 
\begin{align*} 
\int_{A_\sigma} \frac{(f_0 - \phi_{\sigma}*h_{\beta})^2}{\phi_{\sigma}* h_{\beta}} 
&\le \int_{A_{\sigma}} \frac{[f_0(x) - f_0(x)(1 + O(D(x)\sigma^{\beta}))]^2}{K_1 f_0(x)} \\
&\precsim \int_{A_{\sigma}} f_0(x) O(D^2(x)\sigma^{2\beta}) = O(\sigma^{2\beta}).
\end{align*}
To bound the second integral of r.h.s again by Remark 3 in \cite{ghosal1999posterior} we get $\phi_{\sigma}*h_{\beta} \ge \frac{\rho}{3(1+ O(\sigma^{\beta}))}f_0$, so easily we can find a constant $C < 1$ such that $\phi_{\sigma}*h_{\beta} \ge C f_0$. With Lemma \ref{lem:tail} clearly the second and third term can be bounded by $O(\sigma^{2\beta})$.
\end{proof}

\begin{lem}\label{lem:entropy}
Let $\mathbb{H}_1^a$ denote the unit ball of RKHS of the Gaussian process with rescaled parameter $a$ and $\mathbb{B}_1$ be the unit ball of $C[0,1]$. For $r >1$, there exists a constant $K$, such that for $\epsilon < 1/2$,
\begin{eqnarray}\label{eq:entropy}
\log N(\epsilon, \cup_{a \in [0,r]} \mathbb{H}_1^a, \norm{\cdot}_{\infty}) \le Kr\bigg( \log\frac{1}{\epsilon} \bigg)^2. 
\end{eqnarray}
\end{lem}

\begin{proof}
 Since we can write any element of $\mathbb{H}_1^a$ as a function of $\operatorname{Re}(z)$ by  \cite{van2009adaptive} which can be analytically extended to some interval containing $\Omega^{a} = \{z\in\mathbb{C}: |\operatorname{Re}(z)| \le R\}$ with $R = \delta/(6\max(a,1))$, so for any $ h \in \Omega^a$, $|2a\operatorname{Re}(z)| \le |\frac{\delta}{6\max(a,1)} \cdot 2a| = \delta/3$. Consider any $b$ with $|b-a|\le1$, we can show that any element of $\mathbb{H}_1^b$ can be extended analytically to $\Omega^a$ noting that for $z \in \Omega^a$ related to the maximum norm, 
\begin{eqnarray*}
\norm{2b\operatorname{Re}(z)} \le \norm{2a\operatorname{Re}(z)} + \norm{2(b-a)\operatorname{Re}(z)} \le \frac{\delta}{3}+\norm{2\operatorname{Re}(z)} \le \frac{2}{3}\delta.
\end{eqnarray*}  
Therefore, $\mathcal{F}^a$ forms one $\epsilon$-net over $\mathbb{H}_1^b$. We find one set $\Gamma = \{a_i, i = 1, \dots, k\}$ with $k = \lfloor r\rfloor +1$ and $a_k = r$, such that for any $b\in [0,r]$ there exists some $a_i$ satisfying $|b-a_i|\le 1$, so that $\cup_{i\le k}\mathcal{F}^{a_i}$ forms an $\epsilon$-net over $\cup_{a \le r} \mathbb{H}_1^a$. Since the covering number of $\cup_{i\le k}\mathcal{F}^{a_i}$ is bounded by summation of covering number of $\mathcal{F}^{a_i}$, we obtain
\begin{eqnarray*}
\log N(\epsilon, \cup_{a \in [0,r]} \mathbb{H}_1^a, \norm{\cdot}_{\infty}) \le \log \bigg(\sum_{i=1}^k \#(\mathcal{F}^{a_i})\bigg) \le \log(k \cdot \#(\mathcal{F}^r)) \le  Kr\bigg( \log\frac{1}{\epsilon} \bigg)^2.
 \end{eqnarray*} 
To complete the proof, note that the piecewise polynomials are constructed on the partition of $[0,1]$, $\cup_{i\le m}B_i$, where $B_i$'s are disjoint interval with length shorter than $R= \delta/(6\max(a,1))$, so the total number of polynomials is a non-decreasing in $a$. Also we find that when building the mesh grid of the coefficients of polynomials in each $B_i$, both the approximation error and tail estimate are invariant to interval length $R$, therefore we have $\#(\mathcal{F}^a) \le \#(\mathcal{F}^b)$ if $a \le b$, for $a,b \in [0,r]$.  
\end{proof}
\begin{rem}
With larger $a$ we need a finer partition on $[0,1]$ while the grid of coefficients of piecewises polynomial remains the same except the range and the meshwidth will change together along with $a$. Since we can see the element $h$ of RKHS ball as a function of $it$ and with Cauchy formula we can bound the derivatives of $h$ by $C/R^n$, where $|h|^2 \le C^2$. 
\end{rem}

\section{Proof of the main theorem}\label{main}
\noindent \textbf{Proof of Theorem \ref{thm:compact}:}
Following \cite{ghosal2000convergence}, we need to find sequences $\bar{\epsilon}_n,\tilde{\epsilon}_n \to 0$ with
$n\min\{\bar{\epsilon}_n^2,\tilde{\epsilon}_n^2\} \to \infty$  such that there exist constants $C_1, C_2, C_3, C_4> 0$ and sets $\mathcal{F}_n \subset \mathcal{F}$ so that,
\begin{align}
& \log N(\epsilon_n, \mathcal{F}_n, d) \leq C_1n\bar{\epsilon}_n^2 \label{eq1}\\
& \Pi(\mathcal{F}_n^c) \leq C_3\exp\{-n\tilde{\epsilon}_n^2(C_2+4)\}\label{eq2} \\
& \Pi\bigg( f_{\mu, \sigma}:  \int f_0 \log \frac{f_0}{f_{\mu, \sigma}} \leq \tilde{\epsilon}_n^2, \int f_0 \log \bigg(\frac{f_0}{f_{\mu, \sigma}}\bigg)^2 \leq \tilde{\epsilon}_n^2 \bigg) \geq C_4\exp\{-C_2n\tilde{\epsilon}_n^2\} \label{eq3}.
\end{align}
Then we can conclude that for $\epsilon_n = \max\{\bar{\epsilon}_n, \tilde{\epsilon}_n\}$ and sufficiently large $M > 0$, the posterior probability
\begin{eqnarray*}
\Pi_n(f_{\mu, \sigma}: d(f_{\mu, \sigma}, f_0) > M\epsilon_n | Y_1, \ldots, Y_n) \to 0 \, \, \text{a.s.}\, P_{f_0}.
\end{eqnarray*}

We consider the Gaussian process $\mu \sim W^A$ given $A$, with $A$ satisfying Assumption \ref{ass:1p}.

We will first verify (\ref{eq3}) along the lines of \cite{ghosal2007posterior}.
Note that
\begin{eqnarray}\label{eq:H2}
h^{2}(f_0, f_{\mu, \sigma}) \precsim h^{2}(f_0, f_{\mu_0, \sigma}) + h^{2}(f_{\mu_0, \sigma}, f_{\mu, \sigma}).
\end{eqnarray}
Since $f_{\mu_0, \sigma} = \phi_{\sigma}*h_{\beta}$, by Lemma \ref{lem:KL}, one obtains under Assumptions \ref{ass:1} and \ref{ass:2}, \begin{eqnarray}\label{eq:H}
 h^{2}(f_0, f_{\mu_0, \sigma}) \le \int f_0 \log\bigg(\frac{f_0}{f_{\mu_0, \sigma}}\bigg) \precsim O(\sigma^{2\beta}).
\end{eqnarray}

From Lemma \ref{lem:hellinger} and the following remark, we obtain
\begin{eqnarray}
h^{2}(f_{\mu_0, \sigma}, f_{\mu, \sigma}) \precsim \frac{\norm{\mu- \mu_0}_{\infty}^2}{\sigma^2}.
\end{eqnarray}
From Lemma 8 of \cite{ghosal2007posterior}, one has
\begin{eqnarray}\label{eq:KtoH}
\int f_0 \log \bigg(\frac{f_0}{f_{\mu, \sigma}}\bigg)^{i} \leq h^2(f_0, f_{\mu,\sigma})\bigg(1 + \log \norm{\frac{f_0}{f_{\mu, \sigma}}}_{\infty}\bigg)^{i}
\end{eqnarray}
for $i=1,2$.

From (\ref{eq:H2})-(\ref{eq:KtoH}), for any $b \geq 1$ and $\tilde{\epsilon}_n^2 = \sigma_n^{2\beta}$,
\begin{eqnarray*}
\big\{ \sigma \in [\sigma_n, \sigma_n + \sigma_n^b], \norm{\mu - \mu_0}_{\infty} \precsim \sigma_n^{\beta+1} \big\} \subset \\ \bigg\{  \int f_0 \log \frac{f_0}{f_{\mu, \sigma}} \precsim \sigma_n^{2\beta}, \int f_0 \log \bigg(\frac{f_0}{f_{\mu, \sigma}}\bigg)^2 \precsim \sigma_n^{2\beta}\bigg\}.
\end{eqnarray*}

Since $\mu_0 \in C^{\beta+1}[0,1]$, from Section 5.1 of \cite{van2009adaptive},
\begin{eqnarray*}
\mbox{P}(\norm{\mu - \mu_0}_{\infty} \leq 2\delta_n) \geq C_4\exp\{-C_5(1/\delta_n)^{\frac{1}{\beta+1}}\log(\frac{1}{\delta_n})^{2\vee q} \}(C_6/\delta_n)^{p+1/\beta+1},
\end{eqnarray*}
for $\delta_n \to 0$ and constants $C_4, C_5, C_6 > 0$.  Letting $\delta_n = \sigma_n^3$, we obtain
\begin{eqnarray*}
\mbox{P}(\norm{\mu - \mu_0}_{\infty} \leq 2\delta_n) \geq \exp\bigg\{-C_7\bigg(\frac{1}{\sigma_n}\bigg)\log^{2\vee q}\bigg(\frac{1}{\sigma_n^{\beta+1}}\bigg)\bigg\},
\end{eqnarray*}
for some constant $C_7 > 0$.
Since $\sigma \sim IG(a_{\sigma}, b_{\sigma})$,
we have
\begin{eqnarray*}
\mbox{P}(\sigma \in [\sigma_n, 2\sigma_n ]) &=& \frac{b_{\sigma}^{a_{\sigma}}}{\Gamma(a_{\sigma})}\int_{\sigma_n}^{2\sigma_n}x^{-(a_{\sigma}+1)} e^{-b_{\sigma}/x}dx \\ &\geq& \frac{b_{\sigma}^{a_{\sigma}}}{\Gamma(a_{\sigma})}\int_{\sigma_n}^{2\sigma_n} e^{-2b_{\sigma}/x}dx \\
&\geq& \frac{b_{\sigma}^{a_{\sigma}}}{\Gamma(a_{\sigma})}\sigma_n\exp\{-b_{\sigma}/\sigma_n \}\\
&\geq& \exp\{-C_8/\sigma_n \},
\end{eqnarray*}
for some constant $C_8> 0$. Hence
\begin{eqnarray*}
\mbox{P}\{\sigma \in [\sigma_n, 2\sigma_n ], \norm{\mu - \mu_0}_{\infty} \precsim \sigma_n^{\beta+1}\} 
&\geq& \exp\bigg\{-C_7\bigg(\frac{1}{\sigma_n}\bigg)\log^{2\vee q}\bigg(\frac{1}{\sigma_n^{\beta+1}}\bigg)\bigg\}\exp\{-C_8/\sigma_n \}\\
&\ge& \exp\bigg\{-2C_7\bigg(\frac{1}{\sigma_n}\bigg)\log^{2\vee q}\bigg(\frac{1}{\sigma_n^{\beta+1}}\bigg)\bigg\}. 
\end{eqnarray*}
Then (\ref{eq3}) will be satisfied with $\tilde{\epsilon}_n = n^{-\frac{\beta}{2\beta+1}}\log^{t_1}(n)$, where $t_1=\frac{\beta(2\vee q)}{2\beta+1}$ and some $C_9 > 0$.

Next we construct a sequence of subsets $\mathcal{F}_n$ such that \ref{eq1} and \ref{eq2} are satisfied with
$\bar{\epsilon}_n = n^{-\frac{\beta}{2\beta+1}}\log ^{t_2}n$ and $\tilde{\epsilon}_n$ for some global constant $t_2 > 0$.

Letting $\mathbb{B}_1$ denote the unit ball of
$C[0,1]$ and given positive sequences $M_n, r_n$, define
\begin{eqnarray*}
B_n =  \cup_{a < r_n}(M_n \mathbb{H}^a_1) + \bar{\delta}_n\mathbb{B}_1
\end{eqnarray*}
as in \cite{van2009adaptive}, with $\bar{\delta}_n =\bar{\epsilon}_nl_n/K_1, K_1 = 2(2/\pi)^{1/2}$ and let
\begin{eqnarray*}
\mathcal{F}_n = \{f_{\mu, \sigma}: \mu \in B_n, l_n < \sigma < h_n \}.
\end{eqnarray*}
First we need to calculate $N(\bar{\epsilon}_n, \mathcal{F}_n, \norm{\cdot}_1)$. Observe that for $\sigma_2 >\sigma_1 > \frac{\sigma_2}{2}$,
\begin{eqnarray*}
\norm{f_{\mu_1, \sigma_1} - f_{\mu_2, \sigma_2}}_1 \leq \bigg(\frac{2}{\pi}\bigg)^{1/2}\frac{\norm{\mu_1 - \mu_2}_{\infty}}{\sigma_1} + \frac{3(\sigma_2 - \sigma_1)}{\sigma_1}.
\end{eqnarray*}

Taking $\kappa_n =\min\{\frac{\bar{\epsilon}_n}{6}, 1\}$ and $\sigma_m^n = l_n (1+ \kappa_n)^m, m \geq 0$, we obtain a partition of $[l_n, h_n]$ as $l_n=\sigma_0^n < \sigma_1^n < \cdots < \sigma_{m_{n}-1}^n < h_n \leq \sigma_{m_n}^n$ with
\begin{eqnarray}\label{eq:entropy1}
m_n= \bigg(\log \frac{h_n}{l_n}\bigg) \frac{1}{\log( 1+\kappa_n)} +1.
\end{eqnarray}
One can show that $\frac{3(\sigma_m^n -\sigma_{m-1}^n)}{\sigma_{m-1}^n} = 3\kappa_n \leq \bar{\epsilon}_n/2$. Let
$\{\tilde{\mu}_k^n, k=1, \ldots, N(\bar{\delta}_n, B_n, \norm{\cdot}_{\infty})\}$  be
a $\bar{\delta}_n$-net of $B_n$. Now consider the set
\begin{eqnarray} \label{eq:net}
\{(\tilde{\mu}_k^n, \sigma_{m}^n): k=1, \ldots, N(\bar{\delta}_n, B_n, \norm{\cdot}_{\infty}), 0\leq m\leq m_n \}.
\end{eqnarray}
Then for any $f= f_{\mu, \sigma} \in \mathcal{F}_n$, we can find $(\tilde{\mu}_k^n, \sigma_{m}^n)$ such that
$\norm{\mu - \tilde{\mu}_k^n}_{\infty} < \bar{\delta}_n$. In addition, if one has $\sigma \in (\sigma_{m-1}^{n}, \sigma_m^{n}]$,
then
\begin{eqnarray*}
\norm{f_{\mu, \sigma} - f_{\mu_k^n, \sigma^n_m}}_1 \leq \bar{\epsilon}_n.
\end{eqnarray*}
Hence the set in (\ref{eq:net}) is an $\bar{\epsilon}_n$-net of $\mathcal{F}_n$ and its covering number is given by
\begin{eqnarray*}
m_nN(\bar{\delta}_n, B_n, \norm{\cdot}_{\infty}).
\end{eqnarray*}
From the proof of Theorem 3.1 in \cite{van2009adaptive}, for any $M_n, r_n$ with $r_n > a_0$, we obtain
\begin{eqnarray}\label{eq:entropy2}
\log N(2\bar{\delta}_n, B_n, \norm{\cdot}_{\infty}) \leq K_2 r_n \bigg( \log \bigg(\frac{M_n}{\bar{\delta}_n}\bigg)\bigg)^{2}.
\end{eqnarray}
Again from the proof of Theorem 3.1 in \cite{van2009adaptive}, for $r_n > 1$ and for $M_n^2 > 16K_3r_n (\log (r_n / \bar{\delta}_n))^2$, we have
\begin{eqnarray}\label{eq:compact1}
\mbox{P}(W^A \notin B_n) \leq \frac{K_4r_n^p e^{-K_5r_n\log^q r_n}}{K_5\log^q r_n} + \exp\{-M_n^2/8\}
\end{eqnarray}
for constants $K_3, K_4, K_5 > 0$.

Next we calculate $P(\sigma \notin [l_n, h_n])$. Observe that
\begin{eqnarray}
\mbox{P}(\sigma \notin [l_n, h_n ]) &=& P(\sigma^{-1} < h_n^{-1}) + P(\sigma^{-1} > l_n^{-1})\nonumber\\
 &\leq& \sum_{k=\alpha_{\sigma}}^{\infty}\frac{e^{-b_{\sigma}h_n^{-1}}(b_{\sigma}h_n^{-1})^k}{k!} + \frac{b_{\sigma}^{a_{\sigma}}}{\Gamma(a_{\sigma})}\int_{l_n^{-1}}^{\infty} e^{-b_{\sigma}x/2}dx \nonumber \\
&\leq& e^{-a_{\sigma}\log(h_n)} + \frac{b_{\sigma}^{a_{\sigma}}}{\Gamma(a_{\sigma})}e^{-b_{\sigma}l_n^{-1}/2}.\label{eq:compact2}
\end{eqnarray}

Thus with $h_n = O(\exp\{n^{1/{2\beta+1}}(\log n)^{2t_1}\}), l_n = O(n^{-1/{2\beta+1}}(\log n)^{-2t_1}), r_n =O(n^{1/{2\beta+1}}(\log n)^{2t_1}), M_n = O(n^{1/{2\beta+1}}(\log n)^{t_1+1})$,
(\ref{eq:compact1}) and (\ref{eq:compact2}) implies
\begin{eqnarray*}
\Pi(\mathcal{F}_n^c)= \exp\{-K_6n\tilde{\epsilon}^2_n\}
\end{eqnarray*}
for some constant $K_6 > 0$ guaranteeing that (\ref{eq2}) is satisfied with $\tilde{\epsilon}_n = n^{-\beta/{2\beta+1}}(\log n)^{t_1}$.

Also with $\bar{\epsilon}_n = n^{-\beta/{2\beta+1}}(\log n)^{t_1+1}$,  it follows from
(\ref{eq:entropy1}) and (\ref{eq:entropy2}) that
\begin{eqnarray*}
\log N(\bar{\epsilon}_n, \mathcal{F}_n, \norm{\cdot}_1) \leq K_7 n^{1/{2\beta+1}}(\log n)^{2t_1+2}
\end{eqnarray*}
for some constant $K_7 > 0$.

Hence $\max\{\bar{\epsilon}_n, \tilde{\epsilon}_n\} = n^{-\beta/{2\beta+1}}(\log n)^{t_1+1}$.

\section{Discussion}\label{sec:discussion}

Non-linear latent variable models offer a flexible modeling framework in a broad variety of problems and improved practical performance has been demonstrated by \cite{lawrence2004gaussian,lawrence2005probabilistic,lawrence2007hierarchical,ferris2007wifi,kundu2011bayesian} among others. The univariate density estimation model studied here can be extended to multivariate density estimation, latent factor models and density regression problems. 

When the density is compactly supported, the quantile function is a continuous function on $[0, 1]$. Hence one can use standard results on concentration bounds for Gaussian processes \citep{van2008reproducing}. However, for densities supported on $\mathbb{R}$, the results fail as the  corresponding quantile function is unbounded near zero and one. In this case, assumptions on the tail behavior of the true density as well as a careful analysis on the behavior of the corresponding quantile function near boundary are required. We propose to address this problem elsewhere. 

\appendix

\section{Appendix}
\subsection{Proof of Theorem \ref{thm:support}}
\begin{proof}
Let $f_0$ be a density with quantile function $\mu_0$ that satisfies the conditions of Theorem \ref{thm:support}. Observe that $\norm{\mu_0}_1 = \int_{t=0}^1 \abs{\mu_0(t)} dt = \int_{-\infty}^{\infty} \abs{z} f_0(z) dz < \infty$ since $f_0$ has a finite first moment, and thus $\mu_0 \in \mbox{L}_1[0, 1]$. Fix $\epsilon > 0$. We want to show that $\Pi\{ B_{\epsilon}(f_0) \} > 0$, where $B_{\epsilon}(f_0) = \{f ~:~ \norm{f - f_0}_1 < \epsilon\}$.

Note that $\mu_0 \notin C[0, 1]$, so that $\mbox{pr}( \norm{\mu - \mu_0}_{\infty} < \epsilon)$ can be zero for small enough $\epsilon$. The main idea is to find a continuous function $\tilde{\mu_0}$ close to $\mu_0$ in $L_1$ norm and exploit the fact that the prior on $\mu$ places positive mass to arbitrary sup-norm neighborhoods of $\tilde{\mu_0}$. The details are provided below.

Since $\norm{\phi_{\sigma}*f_0 - f_0}_1 \to 0$ as $\sigma \to 0$, find $\sigma_1$ such that $\norm{\phi_{\sigma}*f_0 - f_0}_1 < \epsilon/2$ for $\sigma < \sigma_1$. Pick any $\sigma_0 < \sigma_1$.
Since $C[0, 1]$ is dense in $\mbox{L}_1[0, 1]$, for any $\delta > 0$, we can find a continuous function $\tilde{\mu_0}$ such that $\norm{\mu_0 - \tilde{\mu_0}}_1 < \delta$. Now, $\norm{ f_{\mu, \sigma} - f_{\tilde{\mu_0}, \sigma} }_1 \leq C \norm{ \mu - \tilde{\mu_0} }_1/\sigma$ for a global constant $C$. Thus, for $\delta = \epsilon \sigma_0/4$,
\begin{align*}
\big\{f_{\mu, \sigma} ~:~ \sigma_0 < \sigma < \sigma_1, \norm{\mu - \tilde{\mu_0}}_{\infty} < \delta \big\} \subset \big\{f_{\mu, \sigma} ~:~ \norm{f_0 - f_{\mu, \sigma}}_1 < \epsilon \big\},
\end{align*}
since $\norm{f_0 - f_{\mu, \sigma}}_1 < \norm{f_0 - f_{\mu_0, \sigma}}_1 + \norm{f_{\mu_0, \sigma} - f_{\tilde{\mu_0}, \sigma}}_1 + \norm{f_{\tilde{\mu_0}, \sigma} - f_{\mu, \sigma}}_1$ and $f_{\mu_0, \sigma} = \phi_{\sigma}*f_0$. Thus, $\Pi\{ B_{\epsilon}(f_0) \} > \mbox{pr}(\norm{\mu - \tilde{\mu_0}}_{\infty} < \delta) ~ \mbox{pr}(\sigma_0 < \sigma < \sigma_1) > 0$, since $\Pi_{\mu}$ has full sup-norm support and $\Pi_{\sigma}$ has full support on $[0, \infty)$.
\end{proof}

\subsection{Proof of Lemma \ref{lem:approx}}\label{app:combination}
\begin{proof}
Consider $f_j$ constructed by (\ref{eq:constrction}). When $i=1, f_1 = 2f_{0} - \phi_{\sigma}*f_0$, so the form holds. 
By induction, suppose this form holds for $j > 1$, then 
\begin{eqnarray*}
f_{j+1} &=& f_0 - (\phi_{\sigma}*f_j - f_{j}) \\
&=& f_0 +  \sum_{i=0}^{j}(-1)^{i+1} {j+1 \choose i+1} \phi_{\sigma}^{(i+1)}* f_0 + \sum_{i=0}^{j}(-1)^i {j+1 \choose i+1} \phi_{\sigma}^{(i)}* f_0 \\
&=& (j+2)f_0 +  \sum_{i=1}^{j+1}(-1)^{i} {j+1 \choose i+1} \phi_{\sigma}^{(i)}* f_0 +  \sum_{i=1}^{j}(-1)^{i} {j+1 \choose i} \phi_{\sigma}^{(i)}* f_0\\
&=& (j+2)f_0  + \sum_{i=1}^{j} (-1)^{i} \bigg( {{j+1 \choose i+1} + {j+1 \choose i}} \bigg) \phi_{\sigma}^{(i)}*f_0 + (-1)^{j+1} \phi_{\sigma}^{(i +1)}*f_0 \\
&=& (j+2)f_0 + \sum_{i=1}^{j}(-1)^{i} {j+2 \choose i+1} \phi_{\sigma}^{(i)}* f_0 + (-1)^{j+1} \phi_{\sigma}^{(i +1)}*f_0\\
&=& \sum_{i=0}^{j+1}(-1)^{i} {j+2 \choose i+1} \phi_{\sigma}^{(i)}* f_0.
\end{eqnarray*}
It holds for $j +1$, which completes the proof.
\end{proof}

\bibliographystyle{plain}
\bibliography{Xbib19}
\end{document}